\documentclass[a4paper,10pt]{amsart}
\usepackage{eurosym}
\usepackage{amsmath}
\usepackage{mathrsfs}
\usepackage{amsfonts}
\usepackage{graphicx}
\usepackage{amsfonts}
\usepackage{amssymb}%
\usepackage{amsthm}
\usepackage{titling}
\usepackage{esint}
\usepackage{color}

\usepackage{authblk}

\usepackage{amsmath}

\usepackage[abbrev]{amsrefs}

\usepackage{palatino}

\usepackage{amssymb}

\DeclareFontFamily{U}{mathx}{\hyphenchar\font45}
\DeclareFontShape{U}{mathx}{m}{n}{
      <5> <6> <7> <8> <9> <10>
      <10.95> <12> <14.4> <17.28> <20.74> <24.88>
      mathx10
      }{}
\DeclareSymbolFont{mathx}{U}{mathx}{m}{n}
\DeclareFontSubstitution{U}{mathx}{m}{n}
\DeclareMathAccent{\widecheck}{0}{mathx}{"71}
\DeclareMathAccent{\wideparen}{0}{mathx}{"75}

\providecommand{\U}[1]{\protect\rule{.1in}{.1in}}
\newtheorem{theorem}{Theorem}[section]

\newtheorem{lemma}[theorem]{Lemma}
\newtheorem{proposition}[theorem]{Proposition}
\theoremstyle{definition}
\newtheorem{definition}[theorem]{Definition}

\newtheorem{remark}[theorem]{Remark}

\numberwithin{equation}{section}
\def\cB{\mathcal B}
\def\cI{\mathcal I}
\def\cM{\mathcal M}

\begin{document}

\title{Some remarks on $L^1$ embeddings in the subelliptic setting}
\author[1]{Steven G. Krantz\thanks{sk@math.wustl.edu}}
\affil[1]{Department of Mathematics \\ 
Washington University in St. Louis \\ 
St. Louis, Missouri 63130
}
\author[2]{Marco M. Peloso\thanks{marco.peloso@unimi.it}}
\affil[2]{Dipartimento di Matematica ``F. Enriques''\\
Universit\`a degli Studi di Milano\\
Via C. Saldini 50\\
I-20133 Milano}
\author[3]{Daniel Spector\thanks{dspector@math.nctu.edu.tw}}
\affil[3]{Department of Applied Mathematics\\
 National Chiao Tung University\\
 Hsinchu, Taiwan
}

\maketitle

\begin{abstract}
In this paper we establish an optimal Lorentz estimate for the Riesz potential in the $L^1$ regime in the setting of a stratified group $G$:  Let $Q\geq 2$ be 
the homogeneous dimension of $G$ and $\cI_\alpha$ denote the Riesz potential of order $\alpha$ on $G$.  Then,
for every $\alpha \in (0,Q)$, there
exists a constant $C=C(\alpha,Q)>0$ such that
\begin{align}
\| \cI_\alpha f \|_{L^{Q/(Q-\alpha),1}(G)} \leq C\| X\cI_1 f \|_{L^1(G)}
\end{align}
for distributions $f$ such that $X \cI_1 f \in L^1(G)$, where $X$ denotes the horizontal gradient.
\end{abstract}

\section{Introduction}
A now classical result of S. Sobolev \cite{sobolev} concerns the mapping property of the Riesz potentials on Euclidean space:  For every $\alpha \in (0,d)$ and $p \in (1,d/\alpha)$ there exists a constant 
$C=C(p,\alpha,d)>0$ such that
\begin{align}
\|I_\alpha f\|_{L^{dp/(d-\alpha p)}(\mathbb{R}^d)} \leq C 
\|f\|_{L^p(\mathbb{R}^d)} \label{sobolevineq}
\end{align} 
for all $f \in L^p (\mathbb{R}^d)$.  Here we use $I_\alpha f$ to denote the Riesz potential of order $\alpha$ of the function $f$, defined in Euclidean space by
\begin{align*}
I_\alpha f(x) := \frac{1}{\gamma(\alpha)} \int_{\mathbb{R}^d} \frac{f(y)}{|x-y|^{d-\alpha}}\;dy
\end{align*}
for an appropriate normalization constant $\gamma(\alpha)$.  

While this inequality does not hold in the case $p=1$, there are several possible replacements.  For example, with the assumption $f \in L^1(\mathbb{R}^d)$, one has the weak-type estimate
\begin{align*}
\sup_{t>0} t |\{ |I_\alpha f| >t \}|^{(d-\alpha)/d} \leq C' \|f\|_{L^1(\mathbb{R}^d)},
\end{align*}
which has been pioneered by A. Zygmund in his 1956 paper \cite{Zygmund}.  In order to recover an analogous conclusion to the inequality \eqref{sobolevineq}, one can strengthen the hypothesis, as for example 
in the work of E. Stein and G. Weiss \cite{Stein-Weiss} which implies the inequality
\begin{align*}
\|I_\alpha f\|_{L^{d/(d-\alpha)}(\mathbb{R}^d)} \leq C'' \left( \|f\|_{L^1(\mathbb{R}^d)} + \|\nabla I_1f\|_{L^1(\mathbb{R}^d;\mathbb{R}^d)}\right).
\end{align*}

More recently this estimate has been refined by the third author, in collaboration with Armin Schikorra and Jean Van Schaftingen, in \cite{SSVS} where it was proved that if  $d\geq 2$, then for every $\alpha\in(0,d)$,  there exists a constant $C=C(\alpha,d)>0$ 
such that
\begin{align}
\|I_\alpha f\|_{L^{d/(d-\alpha)}(\mathbb{R}^d)} \leq C
\|\nabla I_1f\|_{L^1(\mathbb{R}^d;\mathbb{R}^d)} \label{L1type}
\end{align}
for all $f \in C^\infty_c(\mathbb{R}^d)$ such that $\nabla I_1f
\in L^1(\mathbb{R}^d;\mathbb{R}^d)$.

An analogue of the inequality \eqref{sobolevineq} has been established for the Heisenberg group by G. Folland and E. Stein \cite{Folland-Stein} and in the more general setting of stratified groups by G. Folland \cite{Folland}, while the Hardy space extension to $p=1$ has been proved in the former setting by the first author in \cite{Krantz} and in the latter by G. Folland and E. Stein in \cite{Folland-Stein-2}.  

One of the main goals of this paper is to extend the inequality \eqref{L1type} to this more general setting.  However, a second aspect of our paper is to prove optimal inequalities on the finer Lorentz scale.  Here let us recall that R. O'Neil's work on convolution estimates in Lorentz spaces implies that one has the following sharpening of the inequality \eqref{sobolevineq} (see \cite{oneil}):   For every $\alpha \in (0,d)$ and $p \in (1,d/\alpha)$ there exists a constant 
$C=C(p,\alpha,d)>0$ such that
\begin{align}\label{oneil}
\|I_\alpha f \|_{L^{q,p}(\mathbb{R}^d)} \leq C \|f\|_{L^p(\mathbb{R}^d)}
\end{align}
for all $f \in L^p (\mathbb{R}^d)$, where 
\begin{align*}
\frac{1}{q} = \frac{1}{p} -\frac{\alpha}{d}.
\end{align*}
The endpoint $p=1$ is forbidden in R. O'Neil's paper \cite{oneil}, and it was work by the third author in \cite{Spector} which obtained the sharpening of the inequality \eqref{L1type} on this Lorentz scale:  Let $d\geq 2$.  For every $\alpha \in (0,d)$, there exists a constant $C=C(\alpha,d)>0$ such that
\begin{align}\label{l1typeestimate}
\| I_\alpha f \|_{L^{d/(d-\alpha),1}(\mathbb{R}^d)} \leq C\| \nabla I_1 f \|_{L^1(\mathbb{R}^d;\mathbb{R}^d)}
\end{align}
for distributions $f$ such that $\nabla I_1f \in L^1(\mathbb{R}^d;\mathbb{R}^d)$.  

It is an exercise in L. Grafakos's book \cite{grafakos} to show that the inequality \eqref{oneil} extends to the setting of a stratified Lie group, as R. O'Neil's paper shows that such an embedding only relies on the structure of the measure spaces and the exponents.  It is natural to conjecture that an analogue to the inequality \eqref{l1typeestimate} holds 
in this setting as well.  Indeed, the main result of our paper is the following.  As it is well known (see Section \ref{preliminaries}), the Lie
algebra $\mathfrak{g}$ of $G$ is generated by the vector fields of the
first stratus $V_1$, and we call them {\em horizontal}.  Having fixed a basis $\{X_1,\dots,X_n\}$ of $V_1$, we denote by
$Xf=(X_1f,\dots,X_nf)$ the horizontal gradient of a function
(distribution) $f$. 
 In this setting there is also a natural analogue of the Riesz potentials, that we denote by $\cI_\alpha$, for $\alpha >0$; see again Section \ref{preliminaries} for precise definitions. \begin{theorem}\label{thm1}
Let $Q\geq 2$ be the homogeneous dimension of $G$.  For every $\alpha \in (0,Q)$, there exists a constant $C=C(\alpha,Q)>0$ such that
\begin{align}
\|   \cI_\alpha   f \|_{L^{Q/(Q-\alpha),1}(G)} 
\leq C\| X  \cI_1 f   \|_{L^1(G)} ,
\end{align}
for distributions $f$ such that $X \cI_1 f \in L^1(G)$, and $X$   denotes the horizontal gradient. 
\end{theorem}
The work of S. Chanillo and J. Van Schaftingen \cite{CVS} obtains results in a similar spirit, for the case of vector-valued differential operators more general than the horizontal gradient, though their results do not obtain the optimal Lorentz space.  In fact, very few optimal Lorentz estimates are known, other than the previously mentioned work of the third author and a recent work of the third author and J. Van Schaftingen \cite{Spector-VanSchaftingen-2018}.

The method of proof has been pioneered in \cite{Spector}, and in particular can be conveniently broken up into five components.
\begin{enumerate}
\item[1.]Establish a pointwise inequality in terms of maximal functions.
\item[2.] Use H\"older's inequality on the Lorentz scale.
\item[3.] Use the boundedness properties of the maximal function.
\item[4.]  Apply to set functions $u = \chi_E$ and reabsorb the bound for the term involving $\chi_E$ by the isoperimetric inequality.
\item[5.]  Use the coarea formula to conclude the general case.
\end{enumerate}

The verification of such a procedure in this more general framework is interesting, as with the abstract definition of the Riesz potential we are led to more natural quantities to estimate.  In place of the Hardy--Littlewood maximal function we work with some maximal functions
associated to the heat kernel $p_t$ for the subLaplacian.  
Thus, for Item 1 in place of the pointwise maximal function
estimate of Maz'ya and Shaposhnikova \cite{Mazya-Shaposhnikova}, we
obtain and utilize the inequality
\begin{align}
|   \cI_\alpha   
X_ju| \leq C \left(\sup_{t>0} | X_j u \ast p_t(x)|\right)^{1-\alpha}
\left(\sup_{t>0} |u \ast \sqrt{t}  (X_j
p_t)\widecheck{\;}  | \right)^{\alpha},
\end{align}
see Lemma \ref{interpolation} below in Section \ref{interpolationsection}.  Item 2 then follows as before, while for Item 3 we now require estimates for these maximal functions.  However, in the setting of stratified groups this is now well-understood:  Cowling, Gaudry, Giulini, and Mauceri have shown in
\cite{CGGM} that one has a weak-$(1,1)$ estimate for the
map
\begin{align*} 
 \cM: f \mapsto \sup_{t>0} | f \ast p_t |(x),
\end{align*}
while the $L^p$-boundedness of
\begin{align*}
\cM_1: f \mapsto \sup_{t>0} | f \ast t^{1/2} \sqrt{t}  (X_j
p_t)\widecheck{\;}   |(x)
\end{align*}
follows from either the theory of grand maximal functions in \cite{Folland-Stein-2} or the same arguments as in \cite{CGGM}, using the Gaussian estimates for the heat kernel $p_t$ -- see e.g. Appendix 1 in \cite{CRT-N} and references therein.  The Lorentz space result thus follows by interpolation.  Finally for Items 4 and 5 the structure present in the stratified Lie group setting is sufficient to argue analogously to the Euclidean case.

The plan of the papers is as follows.  In Section \ref{preliminaries} we recall the requisite preliminaries concerning stratified groups, Lorentz spaces, and functions of bounded variation.  In Section \ref{interpolationsection} we prove a pointwise interpolation inequality for the Riesz potentials in this more general setting.  Finally in Section \ref{mainresult} we prove Theorem \ref{thm1}.

\section{Notation and Preliminaries}\label{preliminaries}

We now define the notion of a stratified group $G$ and the Riesz potential in this setting, as introduced by G. Folland in \cite{Folland}.  

A stratified group $G$ is a nilpotent, simply connected Lie group whose Lie algebra $\mathfrak{g}$ admits
a vector space decomposition such that
\begin{align*}
\mathfrak{g} = V_1\oplus V_2\oplus \cdots \oplus V_m
\end{align*}
such that
\begin{align*}
[V_1,V_k] &= V_{k+1}  \quad 1\leq k<m, \\
[V_1,V_m] &= \{0\}.
\end{align*}
Here we recall that in this setting $\mathfrak{g}$ is a real finite dimensional Lie algebra equipped with a family of dilations $\gamma_r=\operatorname*{exp}(A\ln(r))$ with $A$ a diagonalizable linear transformation of $\mathfrak{g}$ with smallest eigenvalue $1$.  The number $Q= \hbox{\rm trace}(A)$ is the homogeneous dimension with respect to this family of dilations.

For such stratified groups, one can choose a basis
$\{X_1,X_2,\ldots,X_n\}$ of $V_1$ that we fix in the sequel, and  
denote by $Xf = (X_1f,X_2f,\ldots,X_n f)$ the horizontal gradient.  We then define the subLaplacian as the operator 
\begin{align*}
\mathcal{J} := -\sum_{i=1}^n X_i^2 \,.
\end{align*}
With such a definition, $\mathcal{J}$
is hypoelliptic (a result due to H\"ormander which is presented as Proposition (0.1) in \cite{Folland}, p.~161).  In this setting one can define the Riesz kernel as
(see p.~185 in \cite{Folland})
\begin{align*}
 \cI_\alpha 
(x):= \frac{1}{\Gamma(\alpha/2)} \int_0^\infty t^{\alpha/2-1} p_t 
(x) \;dt
\end{align*}
where $p_t(x)$ is the fundamental solution of $\frac{\partial p_t}{\partial t} + \mathcal{J}$, i.e.
\begin{align*}
\frac{\partial p_t}{\partial t} + \mathcal{J}p_t &= 0 \quad \text{ in } G \times \mathbb{R}^+
\end{align*}
and $p_0(x)=\delta_x$, suitably interpreted.  In turn, the Riesz
potential is defined via the convolution on the right 
\begin{align*}
 \cI_\alpha f (x) := f \ast \cI_\alpha(x) = \int f(xy^{-1})
\cI_\alpha(y)\;dy. 
\end{align*}
Here and in the sequel we use $dy$ (or else $dx,dz$, etc.) to denote the bi-invariant Haar measure on $G$, which is the lift of the Lebesgue measure on $\mathfrak{g}$ via the exponential map. Putting these several facts together we find a useful expression for the Riesz potential in 
\begin{align*}
\cI_\alpha 
f = \frac{1}{\Gamma(\alpha/2)} \int_0^\infty t^{\alpha/2-1} f \ast p_t\;dt.
\end{align*}
\begin{remark}
When $p_t$ is the standard heat kernel on Euclidean space, one has $\cI_\alpha 
f =I_\alpha f$, as this formula can be shown to be equivalent to that utilized in the introduction.
\end{remark}

Let us now recall some results concerning the Lorentz spaces $L^{q,r}(G)$.  We follow the convention of R. O'Neil in \cite{oneil}, who proves various results for these spaces under the assumption that $(G,dx)$ is a measure space.  We begin with some definitions related to the non-increasing rearrangement of a function.
\begin{definition}
For $f$ a measurable function on $G$, we define
\begin{align*}
m(f,y):= |\{ |f|>y\}|.
\end{align*} 
As this is a non-increasing function of $y$, it admits a left-continuous inverse, called the non-negative rearrangement of $f$, and which we denote $f^*(x)$.  Further, for $x>0$ we define
\begin{align*}
f^{**}(x):= \frac{1}{x}\int_0^x f^*(t)\;dt.
\end{align*}
\end{definition}
With these basic results, we can now give a definition of the Lorentz spaces $L^{q,r}(G)$.  
\begin{definition}
Let $1<q<+\infty$ and $1\leq r<+\infty$.  We define
\begin{align*}
\|f\|_{L^{q,r}(G)} := \left( \int_0^\infty \left[t^{1/q} f^{**}(t)\right]^r\frac{dt}{t}\right)^{1/r},
\end{align*}
and for $1\leq q \leq+\infty$ and $r=+\infty$
\begin{align*}
\|f\|_{L^{q,\infty}(G)} := \sup_{t>0} t^{1/q} f^{**}(t).
\end{align*}
\end{definition}
For these spaces, one has a duality between $L^{q,r}(G)$ and $L^{q',r'}(G)$ for $1<q<+\infty$ and $1\leq r < +\infty$, where
\begin{align*}
\frac{1}{q}+\frac{1}{q'}&=1\\
\frac{1}{r}+\frac{1}{r'}&=1.
\end{align*}
This implies that one has
\begin{align*}
\| f\|_{L^{q,r}(G)} = \sup \left\{ \left| \int_{G} fg \;dx \right| : g \in L^{q',r'}(G) \ , \ \|g \|_{L^{q',r'}(G)}\leq 1\right\},
\end{align*}
see, for example, Theorem 1.4.17 on p.~52 of \cite{grafakos}.

Let us observe that, with this definition,
\begin{align*}
\|f\|_{L^{1,\infty}(G)} &= \|f\|_{L^1(G)} \\
\|f\|_{L^{\infty,\infty}(G)} &= \|f\|_{L^\infty(G)},
\end{align*}
where the spaces $L^1(G)$ and $L^\infty(G)$ are intended in the usual sense.  It will be important for our purposes to have different endpoints than these, which is only possible through the introduction of a different object.  In particular, 
for $1<q<+\infty$, one has a quasi-norm on the Lorentz spaces $L^{q,r}(G)$ that is equivalent to the norm we have defined.  What is more, this quasi-norm can be used to define the Lorentz spaces without such restrictions on $q$ and $r$.  Therefore let us introduce the following definition.
\begin{definition}
Let $1 \leq q <+\infty$ and $0<r<+\infty$ and we define
\begin{align*}
|||f|||_{\widetilde{L}^{q,r}(G)} :=  \left(\int_0^\infty \left(t^{1/q} f^*(t)\right)^{r} \frac{dt}{t}\right)^{1/r}.
\end{align*}
\end{definition}
Then one has the following result on the equivalence of the quasi-norm on $\widetilde{L}^{q,r}(G)$ and the norm on $L^{q,r}(G)$ (and so in the sequel we drop the tilde):
\begin{proposition}
Let $1<q<+\infty$ and $1\leq r \leq +\infty$.  Then 
\begin{align*}
|||f|||_{\widetilde{L}^{q,r}(G)} \leq \|f\|_{L^{q,r}(G)}\leq q' |||f|||_{\widetilde{L}^{q,r}(G)}.
\end{align*}
\end{proposition}
The proof for $1 \leq r<+\infty$ can be seen by an application of Lemma 2.2 in \cite{oneil}, while the case $r=+\infty$ is an exercise in calculus (see also \cite{Hunt}, equation (2.2) on p.~258).

It will be useful for our purposes to observe an alternative formulation of this equivalent quasi-norm in terms of the distribution function.  In particular, Proposition 1.4.9 in \cite{grafakos} implies the following.
\begin{proposition}
Let $1<q<+\infty$ and $0<r<+\infty$.  Then
\begin{align*}
|||f|||_{L^{q,r}(G)} \equiv q^{1/r} \left(\int_0^\infty \left(t |\{ |f|>t\}|^{1/q}\right)^{r} \frac{dt}{t}\right)^{1/r}.
\end{align*}
\end{proposition}

With either definition one can check the following scaling property that will be useful for our purposes (cf. Remark 1.4.7 in \cite{grafakos}):
\begin{align*}
|||\; |f|^\gamma |||_{L^{q,r}(G)} = |||f|||^\gamma_{L^{\gamma q, \gamma r}(G)}.
\end{align*}

With these definitions, we are now prepared to state H\"older's and Young's inequalities on the Lorentz scale.  
In particular on this scale one has a version of H\"older's inequality (Theorem 3.4 in \cite{oneil}).
\begin{theorem}\label{holder}
Let $f \in L^{q_1,r_1}(G)$ and $g \in L^{q_2,r_2}(G)$, where
\begin{align*}
\frac{1}{q_1}+\frac{1}{q_2}&=\frac{1}{q}<1\\
\frac{1}{r_1}+\frac{1}{r_2}&\geq  \frac{1}{r},
\end{align*}
for some $r \geq 1$.   Then
\begin{align*}
\|fg\|_{L^{q,r}(\mathbb{R}^d)} \leq q'\|f \|_{L^{q_1,r_1}(G)}\|g \|_{L^{q_2,r_2}(G)}
\end{align*}
\end{theorem}
We also have the following generalization of Young's inequality (Theorem 3.1 in \cite{oneil}):
\begin{theorem}\label{young}
Let $f \in L^{q_1,r_1}(G)$ and $g \in L^{q_2,r_2}(\mathbb{R}^d)$, and suppose $1< q<+\infty$ and $1\leq r\leq +\infty$ satisfy
\begin{align*}
\frac{1}{q_1}+\frac{1}{q_2}-1&=\frac{1}{q}\\
\frac{1}{r_1}+\frac{1}{r_2}&\geq \frac{1}{r}.
\end{align*}
Then
\begin{align*}
\|f\ast g\|_{L^{q,r}(G)} \leq 3q \|f \|_{L^{q_1,r_1}(G)}\|g \|_{L^{q_2,r_2}(G)}.
\end{align*}
\end{theorem}
%For these spaces, one has a duality between $L^{q,r}(\mathbb{R}^d)$ and $L^{q',r'}(\mathbb{R}^d)$ for $1<q<+\infty$ and $1\leq r < +\infty$.  This implies that one has
%\begin{align*}
%||| f|||_{L^{q,r}(\mathbb{R}^d)} = \sup \left\{ \left| \int_{\mathbb{R}^d} fg \;dx \right| : g \in L^{q',r'}(\mathbb{R}^d) \;\; |||g |||_{L^{q',r'}(\mathbb{R}^d)}\leq 1\right\},
%\end{align*}
%see, for example, Theorem 1.4.16 on p.~57 of \cite{grafakos}.

Here we utilize certain estimates for functions of bounded variation
and sets of finite perimeter which continue to hold in the setting of
stratified groups (see \cite{GN}). 

We define the space of functions of bounded variation as
\begin{align*}
BV(G):= \left\{ u \in L^1(G) :  \sup_{\Phi\in\cB}  \int_{G} u  
\sum_{j=1}^n  X_j^* \Phi_j \;dx <+\infty \right\},
\end{align*}
where 
\begin{align*}
\cB=  \left\{\Phi \in C^1_c(G;\mathbb{R}^n), \; \| \Phi
   \|_{L^\infty(G;\mathbb{R}^n)} \leq 1\right\} ,
\end{align*} 
and $\{ X_1^*,\dots,X_n^*\}$ denotes the right invariant vector fields which agree with the fixed basis of $V_1$, $\{ X_1,\dots,X_n\}$, at zero.  Then, 
%%If  $X$ denotes any horizontal vector field, then 
the above definition implies that $X_ju$ is a Radon measure with finite
total variation, for $j=1,\dots,n$.  This in turn is equivalent to the
fact that $|Xu|=\big(\sum_{j=1}^n (X_ju)^2 \big)^{1/2}$ is a Radon measure with finite
total variation:  
\begin{align*}
|Xu|(G)= \int_{G} d|Xu| <+\infty.
\end{align*}
We say that a set $E \subset G$ has finite perimeter if $|E|<+\infty$ and $\chi_E \in BV(G)$.  In particular, this implies that
\begin{align*}
|X\chi_E|(G)=  \sup_{\Phi\in\cB}   \int_{G} \chi_E 
\sum_{j=1}^n  
X_j^* \Phi_j \;dx <+\infty 
\end{align*}
For these functions, the coarea formula holds true  (see
\cite{GN} p.~1090):
 \begin{proposition}
 For $u \in BV(G)$, the set $\{u>t\}$ has finite perimeter for almost every $t \in \mathbb{R}$ and for every horizontal vector field
 $X$ we have 
\begin{align*}
|Xu|(G) &= \int_{-\infty}^\infty |X\chi_{\{u>t\}}|(G)\;dt
\end{align*}
\end{proposition}

\section{Pointwise Interpolation Estimates for the Riesz Potentials}\label{interpolationsection}
The goal of this section is establish an analogue of the pointwise interpolation inequality in the setting of stratified groups.  Denoting by $\widecheck h$ the function $\widecheck h(x) =
h(x^{-1})$, in particular we will show that
\begin{lemma}\label{interpolation}
Let $\alpha \in (0,1)$.  There exists a constant $C=C(\alpha)>0$ such that
\begin{align*}
|  \cI_\alpha   X_j u (x)| \leq C \left( \sup_{t>0} | X_j u \ast
  p_t(x)|\right)^{1-\alpha} \left(\sup_{t>0} |u \ast
  \sqrt{t}   (X_jp_t)\widecheck{\;}  |\right)^{\alpha} 
\end{align*}
for all $u \in C^\infty_c(G)$  and $X_j\in V_1$.  
\end{lemma}

Both of these maximal functions admit $L^p$ estimates for
$1<p<+\infty$, which can be found in the work of G. Folland and E. Stein \cite{Folland-Stein-2}, while the first of these maximal functions,
which is the usual maximal function associated with the heat kernel on
the group, has a weak-$(1,1)$ bound because we work on a 
stratified Lie group (see Theorem 4.1 in Cowling, Gaudry, Giulini, and Mauceri \cite{CGGM}).

\proof  %%\begin{proof}
We have
\begin{align*}
 \cI_\alpha  X_j u
&= \frac{1}{\Gamma(\alpha/2)} \int_0^\infty t^{\alpha/2-1} X_ju \ast p_t \;dt \\
&=  \frac{1}{\Gamma(\alpha/2)} \int_0^r t^{\alpha/2-1} X_ju \ast p_t \;dt +  \frac{1}{\Gamma(\alpha/2)} \int_r^\infty t^{\alpha/2-1} X_ju \ast p_t \;dt \\
&=: I(r)+II(r).
\end{align*}
For $I(r)$, we have
\begin{align*}
|I(r)| &= \left|\frac{1}{\Gamma(\alpha/2)} \sum_{n=0}^\infty \int_{r2^{-n-1}}^{r2^{-n}} t^{\alpha/2-1} X_ju \ast p_t \;dt \right| \\
&\leq \frac{1}{\Gamma(\alpha/2)} \sum_{n=0}^\infty \left(r2^{-n-1}\right)^{\alpha/2}\sup_{t>0} | X_j u \ast p_t(x)| \int_{r2^{-n-1}}^{r2^{-n}} \frac{dt}{t} \\
&=  C_1  r^{\alpha/2} \sup_{t>0} | X_j u \ast p_t(x)|,
\end{align*}
where 
\begin{align*}
C_1&:= \frac{\ln(2)}{\Gamma(\alpha/2)} \sum_{n=0}^\infty \left(2^{-n-1}\right)^{\alpha/2}.
\end{align*}
Meanwhile for $II(r)$   we first observe that, 
$\widecheck p_t = p_t$ since $G$ is stratified (see
e.g. \cite{CGGM}). Then, using the identity $X_j u* p_t= -u*(X_j
p_t)\widecheck{\;}$,  we 
have
\begin{align*}
|II(r)| &= \left| \frac{1}{\Gamma(\alpha/2)} \int_r^\infty
  t^{\alpha/2-1} X_ju \ast   p_t   \;dt \right| \\
&\leq \frac{1}{\Gamma(\alpha/2)} \sum_{n=0}^\infty
\int_{r2^n}^{r2^{n+1}} t^{\alpha/2-3/2} |u \ast \sqrt{t} 
 (X_j p_t)\widecheck{\;}  | \;dt \\
&\leq \frac{1}{\Gamma(\alpha/2)}  \sup_{t>0} |u \ast \sqrt{t}
 (X_j
p_t)\widecheck{\;}  |\sum_{n=0}^\infty \left(r2^{n+1}\right)^{\alpha/2-1/2} \int_{r2^n}^{r2^{n+1}}  \;\frac{dt}{t} \\
&= C_2  r^{\alpha/2-1/2}  \sup_{t>0} |u \ast \sqrt{t}  (X_j
p_t)\widecheck{\;}   |
\end{align*}
where
\begin{align*}
C_2:= \frac{\ln(2)}{\Gamma(\alpha/2)} \sum_{n=0}^\infty \left(2^{n+1}\right)^{\alpha/2-1/2}.
\end{align*}

Combining these estimates we find
\begin{align*}
| \cI_\alpha  
X_ju| \leq C_1 r^{\alpha/2} \sup_{t>0} | X_j u \ast p_t(x)| + C_2
r^{\alpha/2-1/2}  \sup_{t>0} |u \ast \sqrt{t}
 (X_j
p_t)\widecheck{\;}  |,
\end{align*}
which with the choice of $r$ such that 
\begin{align*}
 C_1 r^{\alpha/2} \sup_{t>0} | X_j u \ast p_t(x)|=C_2  r^{\alpha/2-1/2}  \sup_{t>0} |u \ast
 \sqrt{t}
 (X_j
p_t)\widecheck{\;}  |
\end{align*}
yields
\begin{align*}
|   \cI_\alpha   
X_ju| \leq C_3 \left(\sup_{t>0} | X_j u \ast p_t(x)|\right)^{1-\alpha}
\left(\sup_{t>0} |u \ast \sqrt{t}  (X_j
p_t)\widecheck{\;}  | \right)^{\alpha}
\end{align*}
with
\begin{align*}
C_3:=2C_1^{1-\alpha} C_2^{\alpha}.    \qed   \bigskip  
\end{align*}

\medskip

%% \end{proof}

\section{Proof of the Main Result}\label{mainresult}

\begin{proof}[Proof of Theorem \ref{thm1}]
We will prove the result for $\alpha \in (0,1)$, the case $\alpha \in [1,Q)$ then follows by R. O'Neil's convolution inequality and the semi-group property of the Riesz potentials (Theorem 3.15 on p.~182 in \cite{Folland}).

We observe that it suffices to prove the existence of a
constant $C>0$ such that, for each $j \in \{ 1,\ldots  n   \}$, one has the inequality
\begin{align}\label{sufficient}
\|  \cI_{\alpha}   X_j u \|_{L^{Q/(Q-\alpha),1}(G)} \leq C\| X u
\|_{L^1(G)} .
\end{align}
 for any vector field $X_j$ of the fixed basis of $V_1$, where
 $Xu$ denotes the horizontal gradient of  $u$.  Indeed, the computation in the proof of Theorem 4.10 on p.~190 of \cite{Folland} shows
\begin{align}\label{identity}
\mathcal{J}^{1/2}u = -\sum_{j=1}^n X_ju \ast \mathcal{J}(K_j \ast \cI_1),
\end{align}
where $K_j=-X^*_j\cI_2$.  Here we recall that kernels $\mathcal{J}(K_j \ast \cI_1)$ in the convolution
\begin{align}\label{operator}
g \mapsto g\ast \mathcal{J}(K_j \ast \cI_1) 
\end{align}
are of type zero.  In particular, the maps \eqref{operator} are bounded firstly on $L^p(G)$, as established in Theorem 4.9 on p.~189, and secondly on $L^{Q/(Q-\alpha),1}(G)$ by interpolation.  Therefore, if we assume the validity of \eqref{sufficient}, these observations and a summation in $j$ would enable us to conclude that
\begin{align*}
 \|  \cI_{\alpha} \mathcal{J}^{1/2} u 
\|_{L^{Q/(Q-\alpha),1}(G)} \leq C\| X u
\|_{L^1(G)}.
\end{align*}
However the conclusion of the theorem then follows by the choice of $u=\cI_1 f$, as one has the identity $\mathcal{J}^{1/2}\cI_1 f=f$.

Therefore we proceed to establish \eqref{sufficient}.  In analogy with the Euclidean case \cite{Spector}, we next argue that it suffices to prove the inequality
\eqref{sufficient} for all $u=\chi_E$ such that $\chi_E \in BV(G)$.  To this end, we first express $u$ as integration of its level sets to obtain the pointwise formula 
\begin{align*}
X_j u = \int_{-\infty}^\infty X_j \chi_{\{u>t\}}\;dt.
\end{align*}
Here we observe that an interchange of the order of integration with the convolution yields
 \begin{align*}
 \cI_{\alpha}X_j u = \int_{-\infty}^\infty  \cI_{\alpha}X_j \chi_{\{u>t\}}\;dt.
\end{align*}
Thus, Minkowski's inequality for integrals, the inequality \eqref{sufficient} applied to the set function $\chi_{\{u>t\}}$, and the coarea formula imply
\begin{align*}
\|  \cI_{\alpha}  X_j u \|_{L^{Q/(Q-\alpha),1}(G)} &\leq
\int_{-\infty}^\infty \|  \cI_{\alpha}  X_j \chi_{\{u>t\}} \|_{L^{Q/(Q-\alpha),1}(G)}\;dt \\
&\leq C\int_{-\infty}^\infty | X\chi_{\{u>t\}}|(G) \;dt \\
&=C \int_G |Xu| \; dx,
\end{align*}
which thus would yield the conclusion of the theorem.

We therefore finally proceed to establish \eqref{sufficient} for $u=\chi_E$.  We begin with the pointwise estimate from Lemma \ref{interpolation}, that if $u \in C^\infty_c(G)$ one has the inequality
\begin{align*}
| \cI_\alpha   X_j u (x)| 
\leq C \left( \sup_{t>0} | X_j u \ast p_t(x)|\right)^{1-\alpha}
\left(\sup_{t>0} |u \ast \sqrt{t}  (X_j
p_t)\widecheck{\;}  |\right)^{\alpha}.
\end{align*}
For convenience of display we contract these two maximal functions as 
\begin{align*}
\mathcal{M}(f) &:=\sup_{t>0} | f \ast p_t|\\
\mathcal{M}_1(f)&:=\sup_{t>0} |f \ast \sqrt{t}  (X_j
p_t)\widecheck{\;} |. 
\end{align*}
By an application of H\"older's inequality, we have
\begin{align*}
\|  \cI_{\alpha}  X_j u \|_{L^{Q/(Q-\alpha),1}(G)} 
\leq C\| (\mathcal{M}(X_ju))^{1-\alpha} \|_{L^{1/(1-\alpha),\infty}(G)}\| (\mathcal{M}_1(u))^{\alpha} \|_{L^{r,1}(G)}
\end{align*}
where $r$ satisfies
\begin{align*}
\frac{1}{Q/(Q-\alpha)} = \frac{1}{1/(1-\alpha)} + \frac{1}{r}.
\end{align*}
Let us observe that a change to the equivalent quasi-norms results in the inequalities
\begin{align*}
\| ((\mathcal{M}(X_ju))^{1-\alpha} \|_{L^{1/(1-\alpha),\infty}(G)} &\leq \frac{1}{\alpha} ||| ((\mathcal{M}(X_ju))^{1-\alpha} |||_{L^{1/(1-\alpha),\infty}(G)},\\
\| (\mathcal{M}_1(u))^{\alpha} \|_{L^{r,1}(G)} &\leq  r' ||| (\mathcal{M}_1(u))^{\alpha} |||_{L^{r,1}(G)}.
\end{align*}
Meanwhile, for these quasi-norms one has the scaling
\begin{align*}
||| ((\mathcal{M}(X_ju))^{1-\alpha} |||_{L^{1/(1-\alpha),\infty}(G)} &= ||| (\mathcal{M}(X_ju) |||^{1-\alpha}_{L^{1,\infty}(G)},\\
||| (\mathcal{M}_1(u))^{\alpha} |||_{L^{r,1}(G)} &= |||\mathcal{M}_1(u) |||^\alpha_{L^{r\alpha,\alpha}(G)}.
\end{align*}
We now recall the boundedness of these two maximal functions, that one has
\begin{align*}
 ||| (\mathcal{M}(X_ju) |||_{L^{1,\infty}(G)} &\leq C_1\| X_j u \|_{L^{1}(G)}, \\
  ||| \mathcal{M}_1(u) |||_{L^{r\alpha,\alpha}(G)} &\leq  C_2||| u |||_{L^{r\alpha,\alpha}(G)}.
\end{align*}
The former follows from Theorem 4.1 in the paper of Cowling, Gaudry, Giulini, and Mauceri \cite{CGGM}, while the latter follows from the boundedness of this maximal function on $L^p$ and interpolation (see Grafakos \cite{grafakos}, p.~56, Theorem 1.4.19).  Thus we deduce that
\begin{align*}
\|  \cI_{\alpha}  X_j u \|_{L^{Q/(Q-\alpha),1}(G)} \leq C'\| X_j u \|_{L^{1}(G)}^{1-\alpha} ||| u |||^\alpha_{L^{r\alpha,\alpha}(G)}.
\end{align*}
From the preceding inequality we may pass to functions for which $X_ju$ is a measure by density, and so we take $u=\chi_E$ and make use of the fact that
\begin{align*}
||| \chi_E |||^\alpha_{L^{r\alpha,\alpha}(G)} = C(\alpha) |E|^{\alpha(1-1/Q)}
\end{align*}
to deduce that
\begin{align*}
\|  \cI_{\alpha}  X_j \chi_E \|_{L^{Q/(Q-\alpha),1}(G)} \leq C''| X_j \chi_E|(G)^{1-\alpha} |E|^{\alpha(1-1/Q)}.
\end{align*}
Next, when we take into account the isoperimetric inequality
\begin{align*}
 |E|^{1-1/Q} \leq \tilde{C}| X \chi_E|(G),
\end{align*}
we see that we have established the inequality
\begin{align*}
\|  \cI_{\alpha}  X_j \chi_E \|_{L^{Q/(Q-\alpha),1}(G)} \leq C'''| X \chi_E|(G),
\end{align*}
which proves the claim and hence the result is demonstrated. 
\end{proof}

\section{Final remarks}

In this work we began the investigation of sharp $L^1$ 
embedding in the subelliptic setting of a subLaplacian on a stratified Lie group.  Of course, it is natural to expect the same result to hold true on a Lie group of polynomial growth.  In this case the Haar measure is still doubling and the results in \cite{CGGM} should extend
to this situation.  More challenging would be the case of a
subLaplacian on a general Lie group, a case recently studied by the second author, see \cite{PV} and \cite{BPTV}.  

%Finally, it would be interesting to characterize the space of distributions $u$ for which $\| X\cI_1 u \|_{L^1(G)}$ is finite.

\section*{Acknowledgements}
 The groundwork for this paper was laid while the second and third named authors were visiting the Department of Mathematics at Washington University in St. Louis. We wish to extend our gratitude to this institution for the hospitality and the pleasant and stimulating working environment provided. The authors would like to thank Gerald Folland for discussions regarding stratified groups and maximal function bounds in this setting.  The second author is supported in part by the 2015 PRIN grant \emph{Real and Complex Manifolds: Geometry, Topology and Harmonic Analysis}.  The third author is supported in part by the Taiwan Ministry of Science and Technology under research grants 107-2918-I-009-003 and 107-2115-M-009-002-MY2.

\end{document}